\documentclass[12pt]{amsart}

\usepackage{amsmath,amssymb, verbatim}

\newtheorem{theorem}{Theorem}[section]

\theoremstyle{definition}

\theoremstyle{remark}

\numberwithin{equation}{section}

\newcommand{\Z}{\mathbf Z}
\newcommand{\Zplus}{\Z_+}
\newcommand{\C}{\mathbf C}

\newcommand{\vectj}[2]{\mbox{$#1_1,#1_2,\ldots ,#1_{#2}$}}  

\newcommand{\qf}[3]{\text{$(#1;\,#2)_{#3}$}}
\newcommand{\qnum}[1]{\text{$[#1]_{q}$}}

\newcommand{\qhyper}[6]{\,{}_{#1}\phi_{#2}\left(\!\!%
            \begin{array}{cc}{#3}\\[-0.1ex]{#4} \end{array}
            \Big|\,{#5};{#6}\right)}

\begin{document}
\title[Symmetric generalization of $q$-{S}turm-{L}iouville problems]{A symmetric generalization of {S}turm-{L}iouville problems in $q$-difference spaces}

\author{I. Area}
\address{Departamento de Matem\'atica Aplicada II,
              E.E. de Telecomunicaci\'on,
              Universidade de Vigo,
              Campus Lagoas-Marcosende,
              36310 Vigo, Spain.}

\email{area@uvigo.es}
\thanks{The work of I. Area has been partially supported by the Ministerio de Ciencia e Innovaci\'on of Spain under grants MTM2009--14668--C02--01 and MTM2012--38794--C02--01, co-financed by the European Community fund FEDER.}

\author{M. Masjed-Jamei}
\address{Department of Mathematics, K.N.Toosi University of Technology, P.O. Box 16315--1618, Tehran, Iran.}

\email{mmjamei@kntu.ac.ir, mmjamei@yahoo.com}
\thanks{The work of M. Masjed-Jamei has been supported by a grant from``Iran National Science Foundation".}


\date{\today}

\begin{abstract}
Classical {S}turm-{L}iouville problems of $q$-difference variables are extended for symmetric discrete functions such that the corresponding solutions preserve the orthogonality property. Some illustrative examples are given in this sense.
\end{abstract}

\maketitle

\section{Introduction}

A regular Sturm-Liouville problem of continuous type is a boundary value problem in the form
\begin{equation}\label{eq:1}
\frac{d}{dx} \left( k(x) \frac{dy_{n}(x)}{dx} \right) + \left(\lambda_{n} \varrho(x)-q(x) \right) y_{n}(x)=0 \qquad (k(x)>0, \varrho(x)>0),
\end{equation}
which is defined on an open interval, say $(a,b)$, and has the boundary conditions
\begin{equation}\label{eq:2}
\alpha_{1} y(a) + \beta_{1} y'(a)=0, \quad \alpha_{2} y(b) + \beta_{2} y'(b)=0,
\end{equation}
where $\alpha_{1}, \alpha_{2}$ and $\beta_{1}, \beta_{2}$, are given constants and $k(x)$, $k'(x)$, $q(x)$, and $\varrho(x)$ in (\ref{eq:1}) are to be assumed continuous for $x \in [a,b]$. In this sense, if one of the boundary points $a$ and $b$ is singular (i.e. $k(a) = 0$ or $k(b) = 0$), the problem is called a singular {S}turm-{L}iouville problem of continuous type.

Let $y_{n}$ and $y_{m}$ be two eigenfunctions of equation (\ref{eq:1}). According to {S}turm-{L}iouville theory \cite{MR922041}, they are orthogonal with respect to the weight function $\varrho(x)$ under the given conditions (\ref{eq:2}) so that we have
\begin{equation}\label{eq:3}
\int_{a}^{b} \varrho(x) y_{n}(x) y_{m}(x) dx =
\left( \int_{a}^{b} \varrho(x) y_{n}^{2}(x) dx \right) \delta_{n,m}=
\| y_{n} \|_{2}^{2} \begin{cases} 0 & n \neq m, \\ 1 & n=m. \end{cases}
\end{equation}

Many important special functions in theoretical and mathematical physics are solutions of a regular or singular {S}turm-{L}iouville problem that satisfy the orthogonality condition (\ref{eq:3}). For instance, the associated Legendre functions \cite{MR1810939}, Bessel functions \cite{MR922041}, Fourier trigonometric sequences \cite{MR0105586}, ultraspherical functions \cite{MR922041} and Hermite functions \cite{MR922041} are some specific continuous samples. Most of these functions are symmetric (i.e. $\phi_{n}(-x)=(-1)^{n} \phi_{n}(x)$) and have found valuable applications in physics and engineering. Hence, if we can somehow extend these examples symmetrically and preserve their orthogonality property, it seems that we will be able to find new applications, which logically extend the previous established applications. Recently in \cite{MR2374588}, this matter has been done for continuous variables and the classical equation (\ref{eq:1}) has been symmetrically extended in the following form
\begin{equation}\label{eq:4}
A(x) \phi_{n}''(x) + B(x) \phi_{n}'(x) + \left( \lambda_{n} C(x) + D(x) + \sigma_{n} E(x)\right) \phi_{n}(x)=0,
\end{equation}
where $A(x)$, $D(x)$, $E(x)$ and $(C(x)>0)$ are even functions, $B(x)$ is an odd function and 
\begin{equation}\label{eq:sigma}
\sigma_{n}=\frac{1-(-1)^{n}}{2} = \begin{cases}
0 & n \text{ even}, \\
1 & n \text{ odd}.
\end{cases} 
\end{equation}
It has been proved in \cite{MR2374588} that under some specific conditions, the symmetric solutions of equation (\ref{eq:4}) are orthogonal and preserve the orthogonality interval, in other words:
\begin{theorem}\cite{MR2374588}\label{th:1}
Let $\phi_{n}(-x)=(-1)^{n} \phi_{n}(x)$ be a sequence of symmetric functions that satisfies the differential equation (\ref{eq:4}), where $\{Ê\lambda_{n} \}_{n}$ is a sequence of constants. If $A(x)$, $(C(x) > 0)$, $D(x)$ and $E(x)$ are even real  functions and $B(x)$ is odd then
\begin{equation*}
\int_{-\nu}^{\nu} P^{*}(x) \phi_{n}(x) \phi_{m}(x) dx = \left( \int_{-\nu}^{\nu}P^{*}(x) \phi_{n}^{2}(x) dx \right) \delta_{n,m},
\end{equation*}
where
\begin{equation}\label{eq:6}
P^{*}(x)= C(x) {\exp} \left( \int \frac{B(x)-A'(x)}{A(x)} dx \right) = \frac{C(x)}{A(x)} {\exp} \left( \int \frac{B(x)}{A(x)} dx \right).
\end{equation}
Of course, the weight function defined in (\ref{eq:6}) must be positive and even on $[-\nu,\nu]$ and the function
\begin{equation*}
A(x)K(x)=A(x) {\exp} \left( \int \frac{B(x)-A'(x)}{A(x)} dx\right) = {\exp} \left( \int \frac{B(x)}{A(x)} dx\right),
\end{equation*}
must vanish at $x = \nu$ , i.e. $A(\nu) K(\nu) = 0$. In this way, since $K(x) = P^{*}(x) /C(x)$ is an even function so $A(-\nu) K(-\nu) = 0$ automatically.
\end{theorem}

By using theorem \ref{th:1}, many symmetric special functions of continuous type have been generalized in \cite{MR2270049,MR2374588,MR2421844,MR2601301,MR2467682,MR2743534,1220.33011}. Recently in \cite{MASJEDAREA} we have generalized usual Sturm-Liouville problems with symmetric solutions in discrete spaces on the linear lattice $x(s)=s$, and introduced a basic class of symmetric orthogonal polynomials of a discrete variable with four free parameters \cite{MASJEDAREA2}. The main aim of this paper is to prove that the aforesaid extension also holds in a $q$-difference case and for a homogeneous second-order $q$-difference equation.

\section{A symmetric generalization of {S}turm-{L}iouville problems in $q$-difference spaces}

Before stating the results, we should consider some preliminaries and notations.

Let $\mu \in {\mathbf{C}}$ be fixed. A set $A \subseteq {\mathbf{C}}$ is called a $\mu$-geometric set if for $x \in A$, $\mu x \in A$. Let $f$ be a function defined on a $q$-geometric set $A \subseteq {\mathbf{C}}$. The $q$-difference operator is defined by
\begin{equation*}
D_{q}f(x)=\frac{f(qx)-f(x)}{(q-1)x}, \qquad x \in A \setminus \{0\}.
\end{equation*}
If $0 \in A$, we say that $f$ has the $q$-derivative at zero if the limit
\begin{equation*}
\lim_{n \to \infty} \frac{f(x q^{n})-f(0)}{xq^{n}} \qquad (x \in A),
\end{equation*}
exists and does not depend on $x$. We then denote this limit by $D_{q} f(0)$.

We shall also need the $q$-integral (the inverse of the $q$-derivative operator) introduced by J. Thomae \cite{Thomae1969} and F.H. Jackson \cite{JACKSON1910} ---see also \cite{MR2128719,MR2191786,MR2656096}--- which is defined as
\begin{equation}\label{eq:11}
\int_{0}^{x} f(t) \text{d}_qt = x(1-q)\sum_{n=0}^{\infty} q^n f(q^n x)\,, \qquad (x \in A),
\end{equation}
provided that the series converges, and for the interval $[a,b]$ we have based on (\ref{eq:11}) that
\begin{equation}\label{eq:10}
\int_{a}^{b} f(t) d_{q}t = \int_{0}^{b} f(t) d_{q}t - \int_{0}^{a} f(t) d_qt, \qquad (a,b \in A).
\end{equation}
Relations (\ref{eq:11}) and (\ref{eq:10}) directly yield
\begin{equation}\label{eq:12}
\int_{-b}^{b} f(t) d_{q}t = b(1-q) \sum_{n=0}^{\infty} q^{n} \left( f(bq^{n})+f(-b q^{n}) \right), \qquad (b \in A).
\end{equation}
This means that if $f$ is an odd function, then  $\displaystyle{\int_{-b}^{b} f(t) d_{q}t=0}$. Moreover, if $b \to \infty$, (\ref{eq:12}) changes to
\begin{equation*}
\int_{-\infty}^{\infty} f(t) d_{q}t =(1-q) \sum_{n=-\infty}^{\infty} q^{n} \left( f(q^{n}) + f(-q^{n})\right).
\end{equation*}

A function $f$ which is defined on a $q$-geometric set $A$ with $0 \in A$ is said to be $q$-regular at zero if $\displaystyle{\lim_{n \to \infty} f(xq^{n})=f(0)}$ for every $x \in A$. The rule of $q$-integration by parts is denoted by
\begin{equation}\label{eq:14}
\int_{0}^{a} g(x) D_{q}f(x) d_{q}x=(fg)(a)-\lim_{n \to \infty} (fg)(aq^{n}) - \int_{0}^{a} D_{q} g(x)f(qx)d_{q}x.
\end{equation}
If $f,g$ are $q$-regular at zero, the $\displaystyle{\lim_{n \to \infty} (fg)(aq^{n})}$ on the right-hand side of (\ref{eq:14}) can be replaced by $(fg)(0)$.

For $0 < R \leq \infty$ let $\Omega_{R}$ denote the disc $\{ z \in {\mathbf{C}} \, : \, \vert z \vert < R\}$. The $q$-analogue of the fundamental theorem says:
Let $f: {\Omega_{R}} \to {\mathbf{C}}$ be $q$-regular at zero and $\theta \in \Omega_{R}$ be fixed. Define
\begin{equation*}
F(x)=\int_{\theta}^{x} f(t) d_{q}t \qquad (x \in \Omega_{R}).
\end{equation*}
Then, the function $F$ is $q$-regular at zero, $D_{q}F(x)$ exists for any $x \in \Omega_{R}$ and $D_{q}F(x)=f(x)$. Conversely, if $a,b \in \Omega_{R}$ then
\begin{equation*}
\int_{a}^{b} D_{q}f(t) d_{q}t=f(b)-f(a).
\end{equation*}

The function $f$ is $q$-integrable on $\Omega_{R}$ if $\displaystyle{ \vert f(t) \vert d_{q}t}$ exists for all $x \in \Omega_{R}$.

\begin{theorem}
Let $\phi_{n}(x;q)=(-1)^{n} \phi_{n}(-x;q)$ be a sequence of symmetric functions that satisfies the $q$-difference equation 
\begin{multline}\label{eq:17}
A(x) D_{q} D_{q^{-1}} \phi_{n}(x;q) + B(x) D_{q} \phi_{n}(x;q) \\ 
+ \left( \lambda_{n,q} C(x) + D(x) + \sigma_{n} E(x) \right) \phi_{n}(x;q)=0,
\end{multline}
where $A(x)$, $B(x)$, $C(x)$, $D(x)$ and $E(x)$ are independent functions, $\sigma_{n}$ is defined in (\ref{eq:sigma})
and $\lambda_{n,q}$ is a sequence of constants. If $A(x)$, $(C(x)>0)$, $D(x)$ and $E(x)$ are even functions and $B(x)$ is odd, then
\begin{equation*}
\int_{-\alpha}^{\alpha} W^{*}(x;q) \phi_{n}(x;q) \phi_{m}(x;q) d_{q}x =
\left( \int_{-\alpha}^{\alpha} W^{*}(x;q) \phi_{n}^{2}(x;q) d_{q}x  \right) \delta_{n,m},
\end{equation*}
where 
\begin{equation}\label{eq:19}
W^{*}(x;q)=C(x)W(x;q),
\end{equation}
and $W(x;q)$ is solution of the Pearson $q$-difference equation
\begin{equation}\label{eq:20}
D_{q} \left(A(x) W(x;q) \right) = B(x) W(x;q),
\end{equation}
which is equivalent to
\begin{equation*}
\frac{W(qx;q)}{W(x;q)}=\frac{(q-1)x B(x)+A(x)}{A(qx)}.
\end{equation*}
Of course, the weight function defined in (\ref{eq:19}) must be positive and even and $A(x)W(x;q)$ must vanish at $x=\alpha$.
\end{theorem}
\begin{proof}
If the difference equation (\ref{eq:17}) is written in a self-adjoint form, then
\begin{multline}\label{eq:22}
D_{q} \left[A(x)W(x;q) D_{q^{-1}} \phi_{n}(x;q) \right]\\ + \left( \lambda_{n,q} C(x) + D(x) + \sigma_{n} E(x) \right) W(x;q) \phi_{n}(x;q)=0,
\end{multline}
and for $m$ we similarly have
\begin{multline}\label{eq:23}
D_{q} \left[A(x)W(x;q) D_{q^{-1}} \phi_{m}(x;q) \right] \\ + \left( \lambda_{m,q} C(x) + D(x) + \sigma_{m} E(x) \right) W(x;q) \phi_{m}(x;q)=0.
\end{multline}
By multiplying (\ref{eq:22}) by $\phi_{m}(x;q)$ and (\ref{eq:23}) by $\phi_{n}(x;q)$ and subtracting each other we get
\begin{multline}\label{eq:24}
\phi_{m}(x;q) D_{q} \left[ A(x) W(x) D_{q^{-1}} \phi_{n}(x;q) \right] - \phi_{n}(x;q) D_{q} \left[ A(x) W(x) D_{q^{-1}} \phi_{m}(x;q) \right] \\
 + \left( \lambda_{n,q}-\lambda_{m,q} \right)C(x)W(x;q) \phi_{n}(x;q)\phi_{m}(x;q)\\
+\frac{(-1)^{m}-(-1)^{n}}{2}E(x)W(x;q) \phi_{n}(x;q)\phi_{m}(x;q)=0.
\end{multline}
A simple but important idea can appear here: ``The $q$-integration of any odd integrand over a symmetric interval is equal to zero". Therefore, $q$-integrating on both sides of (\ref{eq:24}) over the symmetric interval $[-\alpha,\alpha]$ yields
\begin{multline}\label{eq:25}
\int_{-\alpha}^{\alpha} \phi_{m}(x;q) D_{q} \left( A(x)W(x;q) D_{q^{-1}} \phi_{n}(x;q) \right) d_{q} x \\
 - \int_{-\alpha}^{\alpha} \phi_{n}(x;q) D_{q} \left( A(x)W(x;q) D_{q^{-1}} \phi_{m}(x;q) \right) d_{q} x  \\ + \left( \lambda_{n,q}-\lambda_{m,q} \right) \int_{-\alpha}^{\alpha}  C(x)W(x;q) \phi_{n}(x;q)\phi_{m}(x;q) d_{q}x  \\
+\frac{(-1)^{m}-(-1)^{n}}{2} \int_{-\alpha}^{\alpha}  E(x)W(x;q) \phi_{n}(x;q)\phi_{m}(x;q) d_{q}x=0.
\end{multline}
Now, by using the rule of $q$-integration by parts, relation (\ref{eq:25}) is transformed to
\begin{multline}\label{eq:26}
\left[ A(x) W(x;q) \phi_{m}(x;q) D_{q^{-1}} \phi_{n}(x;q) \right]_{-\alpha}^{\alpha} \\
- \int_{-\alpha}^{\alpha} A(qx)W(qx;q) D_{q^{-1}} \phi_{n}(qx;q) D_{q} \phi_{m}(x;q) d_{q} x \\
-\left[ A(x) W(x;q) \phi_{n}(x;q) D_{q^{-1}} \phi_{m}(x;q) \right]_{-\alpha}^{\alpha} \\
+ \int_{-\alpha}^{\alpha} A(qx)W(qx;q) D_{q^{-1}} \phi_{m}(qx;q) D_{q} \phi_{n}(x;q) d_{q} x \\
+ \left( \lambda_{n,q}-\lambda_{m,q} \right) \int_{-\alpha}^{\alpha}  C(x)W(x;q) \phi_{n}(x;q)\phi_{m}(x;q) d_{q}x  \\
+\frac{(-1)^{m}-(-1)^{n}}{2} \int_{-\alpha}^{\alpha}  E(x)W(x;q) \phi_{n}(x;q)\phi_{m}(x;q) d_{q}x=0.
\end{multline}
Since
\begin{equation*}
D_{q^{-1}} f(qx)=D_{q}f(x),
\end{equation*}
relation (\ref{eq:26}) is simplified as
\begin{multline}\label{eq:28}
\left[ A(x) W(x;q) \left( \phi_{m}(x;q) D_{q^{-1}} \phi_{n}(x;q) - \phi_{n}(x;q) D_{q^{-1}} \phi_{m}(x;q) \right) \right]_{-\alpha}^{\alpha} \\
+ \left( \lambda_{n,q}-\lambda_{m,q} \right) \int_{-\alpha}^{\alpha}  C(x)W(x;q) \phi_{n}(x;q)\phi_{m}(x;q) d_{q}x  \\
+\frac{(-1)^{m}-(-1)^{n}}{2} \int_{-\alpha}^{\alpha}  E(x)W(x;q) \phi_{n}(x;q)\phi_{m}(x;q) d_{q}x=0.
\end{multline}
On the other hand, $W(x;q)$ is a symmetric solution for the Pearson $q$-difference equation (\ref{eq:20}). Hence, if in (\ref{eq:28}) we take
\begin{equation*}	
A(-\alpha)W(-\alpha;q)=A(\alpha)W(\alpha;q)=0,
\end{equation*}
then to prove the orthogonality property it remains to show that
\begin{equation*}
F(m,n)=\frac{(-1)^{m}-(-1)^{n}}{2} \int_{-\alpha}^{\alpha}  E(x)W(x;q) \phi_{n}(x;q)\phi_{m}(x;q) d_{q}x=0.
\end{equation*}
For this purpose, four cases should be considered for values $m$, $n$, which are respectively as 
follows:
\begin{enumerate}
\item If both $m$ and $n$ are even (or odd), then $F(n,m)=0$ because we have $F(2i,2j)=F(2i+1,2j+1)=0$.

\item If one of the two mentioned values is odd and the other one is even (or conversely) then
\begin{equation}\label{eq:31}
F(2i,2j+1)=\int_{-\alpha}^{\alpha} E(x)W(x;q) \phi_{2j+1}(x;q) \phi_{2i}(x;q) d_{q}x.
\end{equation}
\end{enumerate}

Since $E(x)$, $W(x;q)$ and $\phi_{2i}(x;q)$ are assumed to be even functions and $\phi_{2j+1}(x;q)$ is odd in (\ref{eq:31}), its integrand would be an odd function and therefore $F(2i,2j+1)=0$. This results similarly holds for the case $m=2i+1$ and $n=2j$, i.e. $F(2i+1,2j)=0$.
$\text{ }$ \hfill \end{proof}

\section{Some illustrative examples}

In this section we consider 3 examples of the extended equation (\ref{eq:17}) whose solutions are symmetric $q$-orthogonal polynomials.

For this purpose, let us define some notations related to $q$-polynomials. The $q$-shifted factorial is defined by
\begin{equation*}
 \qf xqn=\prod_{j=0}^{n-1}(1-q^jx),\qquad n=0,1,\ldots,
\end{equation*}
the $q$-number by
\begin{equation*}
\qnum{z}=\frac{q^{z}-1}{q-1}, \quad z \in {\mathbf{C}},
\end{equation*}
and the {basic hypergeometric series} is defined by
\begin{equation*}
 \qhyper rs {\vectj ar}{\vectj bs}{q}{z}=
 \sum_{k=0}^{\infty}\frac{\qf{a_1}qk\ldots \qf{a_r}qk}
                                 {\qf qqk\qf{b_1}qk\ldots \qf{b_s}qk}
               \left((-1)^kq^{\binom{k}{2}}\right)^{1+s-r}z^k.
\end{equation*}
Here $r,\,s\in\Zplus$ and $\vectj ar$, $\vectj bs$, $z$ $\in \C$.
In order to have a well--defined series the condition $ \vectj
bs \neq q^{-k}$ ($k=0,1,\ldots $) is required.

\subsection{Example 1}

Let us consider the $q$-difference equation
\begin{multline*}
x^{2}(1-x^{2}) D_{q} D_{q^{-1}} \phi_{n}(x;q)
+qx \left(-(q^{2}+q+1) x^2+q+1 \right) D_{q} \phi_{n}(x;q) 
\\+\left(\qnum{n} \left( q (q^{2}+q+1)-\qnum{1-n} \right) x^{2} -q (1+q) \sigma_{n} \right) \phi_{n}(x;q) =0,
\end{multline*}
as a special case of (\ref{eq:17}). If we set
\begin{equation*}
\phi_{n}(x;q)=\sum_{j=0}^{\infty} a_{j}(n) x^{j},
\end{equation*}
then a solveble recurrence relation for the coefficients $\{a_{j}(n)\}_{j=0}^{\infty}$ is derived giving rise eventually to the following representation
\begin{equation}\label{eq:icod}
\phi_{n}(x;q)=x^{\sigma_{n}} \, \qhyper{2}{1}
{q^{\sigma_{n}-n},q^{n+\sigma_{n}-1} \left(q^4-q+1\right)}
{q^{2 \sigma_{n}+1} \left(q^3-q+1\right)}
{q^{2}}
{q^2 x^2}.
\end{equation}

From Section 2, it is known that the above sequence satisfies the orthogonality relation
\begin{equation*}
\int_{-1}^{1} W_{1}^{*}(x;q) \phi_{n}(x;q) \phi_{m}(x;q) d_{q}x = 
\left( \int_{-1}^{1} W_{1}^{*}(x;q) \phi_{n}^{2}(x;q) d_{q}x \right) \delta_{n,m},
\end{equation*}
in which $W_{1}^{*}(x;q)=C(x)W_{1}(x;q)$ is the main weight function and $W_{1}(x;q)$ satisfies the equation
\begin{equation}\label{eq:360}
\frac{W_{1}(qx;q)}{W_{1}(x;q)}=\frac{(q^{4}-q+1) x^2-q^3+q-1}{q^2 \left(q^2 x^2-1\right)}.
\end{equation}
Up to a periodic function, a solution of the equation (\ref{eq:360}) is in the form
\begin{equation*}
W_{1}(x;q)=\frac{
\qf{q^2 x^2}{q^{2}}{\infty}
   \left(q^3-q+1\right)^{\frac{\log \left(x^2\right)}{2 \log(q)}}}
 {x^2 \qf{\frac{q^{4}-q+1}{q^3-q+1}x^2}{q^{2}}{\infty} }=W_{1}(-x;q).
\end{equation*}
In this sense, note that
\begin{equation*}
\lim_{q \uparrow 1} W_{1}^{*}(x;q)=\lim_{q \uparrow 1} x^{2} W_{1}(x;q)=\frac{x^{2}}{\sqrt{1-x^{2}}},
\end{equation*}
which gives the weight function of the fifth kind Chebyshev polynomials \cite{MR2270049}.

To compute the norm square value of the symmetric polynomials (\ref{eq:icod}), we can use Favard's theorem \cite{MR0481884}, which says if $\{P_{n}(x;q)\}$ satisfies the recurrence relation
\begin{equation*}
x P_{n}(x;q) = A_{n} P_{n+1}(x;q) + B_{n} P_{n}(x;q) + C_{n} P_{n-1}(x;q), \qquad n=0,1,2,\dots,
\end{equation*}
where $P_{-1}(x;q)=0$, $P_{0}(x;q)=1$, $A_{n}$, $B_{n}$, $C_{n}$ real and $A_{n} C_{n+1}>0$ for $n=0,1,2,\dots$, then there exists a weight function $W^{*}(x;q)$ so that
\begin{equation*}
\int_{-\alpha}^{\alpha} W^{*}(x;q) P_{n}(x;q) P_{m}(x;q) d_{q}x=
\left( \prod_{i=0}^{n-1} \frac{C_{i+1}}{A_{i}} \int_{-\alpha}^{\alpha} W^{*}(x;q) d_{q}x \right) \delta_{n,m}.
\end{equation*}
It is clear that the Favard's theorem also holds for the monic type of symmetric $q$-polynomials in which $A_{n}=1$ and $B_{n}=0$. So, if $\bar{\phi}_{n}(x;q)$ is consider as the monic form of the symmetric $q$-polynomials (\ref{eq:icod}), then after some calculations they would satisfy the following three term recurrence relation
\begin{equation}\label{eq:45}
\bar{\phi}_{n+1}(x;q)=x \bar{\phi}_{n}(x;q) - \gamma_{n} \bar{\phi}_{n-1}(x;q), \qquad (\text{with } \bar{\phi}_{0}(x;q)=1, \quad \bar{\phi}_{1}(x;q)=x),
\end{equation}
in which
\begin{equation*}
\gamma_{2m}=\frac{q^{2 m+1} \left(q^{2 m}-1\right)
   \left(\left(q^4-q+1\right) q^{2 m}+\left(-q^3+q-1\right)
   q^2\right)}{\left(q^4-q+1\right)^2 q^{8
   m}-\left(q^2+1\right) \left(q^4-q+1\right) q^{4 m+1}+q^4},
\end{equation*}
and
\begin{equation*}
\gamma_{2m+1}=\frac{q^{2 m} \left(\left(q^3-q+1\right) q^{2 m+1}-1\right)
   \left(\left(q^4-q+1\right) q^{2
   m}-q\right)}{\left(q^4-q+1\right)^2 q^{8
   m+1}-\left(q^2+1\right) \left(q^4-q+1\right) q^{4 m}+q}.
\end{equation*}

Note that
\begin{equation*}
\lim_{q \uparrow 1} \gamma_{n}=\frac{(4 n+2) \sigma_{n}+(n-1) n}{4 n (n+1)},
\end{equation*}
which coincides with \cite[Eq. (61.1)]{MR2270049}.

Therefore, the norm square value of the monic type of the $q$-polynomials (\ref{eq:icod}) takes the form
\begin{equation*}
\int_{-1}^{1} \bar{\phi}_{n}^{2}(x;q) W_{1}^{*}(x;q) d_{q}x= d_{n}^{2} \int_{-1}^{1} 
\frac{
\qf{q^2 x^2}{q^{2}}{\infty}
   \left(q^3-q+1\right)^{\frac{\log \left(x^2\right)}{2 \log(q)}}}
 { \qf{\frac{q^{4}-q+1}{q^3-q+1}x^2}{q^{2}}{\infty} } d_{q}x,
\end{equation*}
where
\begin{multline*}
d_{2m}^{2}=\frac{(q-1)^2 \left(q^2+q+1\right) \left(q^3+q^2+q-1\right) q^{m (2
   m-1)-2} \left(q^3-q+1\right)^{m+1}}{ \left(q
   \left(q^4-q+1\right);q^4\right){}_m} \\
   \times 
   \frac{\left(q^2;q^2\right){}_m
   \left(q^3-1+\frac{1}{q};q^2\right){}_m \left(q
   \left(q^3-q+1\right);q^2\right){}_m
   \left(\frac{q^4-q+1}{q^5-q^3+q^2};q^2\right){}_{m+1}}{\left(
   q^4-q^2+1\right)
   \left(q-\frac{1}{q^2}+\frac{1}{q^3};q^4\right){}_{m+1}
   \left(q^3-1+\frac{1}{q};q^4\right){}_m
   \left(q^3-1+\frac{1}{q};q^4\right){}_{m+1}} ,
\end{multline*}
and
\begin{multline*}
d_{2m+1}^{2}=\frac{(q-1)^2 \left(q^2+q+1\right) \left(q^3+q^2+q-1\right) q^{2
   m^2+m-2} \left(q^3-q+1\right)^{m+1}}{\left(q \left(q^4-q+1\right);q^4\right){}_{m+1}} \\
   \times \frac{\left(q^2;q^2\right){}_m
   \left(q^3-1+\frac{1}{q};q^2\right){}_{m+1} \left(q
   \left(q^3-q+1\right);q^2\right){}_{m+1}
   \left(\frac{q^4-q+1}{q^5-q^3+q^2};q^2\right){}_{m+1}}{\left(
   q^4-q^2+1\right)
   \left(q-\frac{1}{q^2}+\frac{1}{q^3};q^4\right){}_{m+1}
   \left(\left(q^3-1+\frac{1}{q};q^4\right){}_{m+1}\right){}^2
   }.
\end{multline*}

\subsection{Example 2}

Consider the $q$-difference equation
\begin{multline}\label{eq:368}
x^{2}(1-x^2) D_{q} D_{q^{-1}} \phi_{n}(x;q)
+qx\left(  -\qnum{5} x^{2}+q+1 \right) D_{q} \phi_{n}(x;q) 
\\+\left(
\qnum{n} (q \qnum{5}- \qnum{1-n} ) x^{2} - q (q+1) \sigma_{n}
 \right) \phi_{n}(x;q) =0,
\end{multline}
as a special case of (\ref{eq:17}).

Following the approach of example 1, we can obtain the polynomial solution of equation (\ref{eq:368}) eventually
\begin{equation*}
\phi_{n}(x;q)=
x^{\sigma_{n}} 
\qhyper{2}{1}
{q^{\sigma_{n}-n},q^{n+\sigma_{n}-1} \left(q^6-q+1\right)}
{q^{2 \sigma_{n}+1}\left(q^3-q+1\right)}
{q^2}
{q^2 x^2}.
\end{equation*}

Again, this sequence satisfies an orthogonality relation in the form
\begin{equation*}
\int_{-1}^{1} W_{2}^{*}(x;q) \phi_{n}(x;q) \phi_{m}(x;q) d_{q}x = 
\left( \int_{-1}^{1} W_{2}^{*}(x;q) \phi_{n}^{2}(x;q) d_{q}x \right) \delta_{n,m},
\end{equation*}
where $W_{2}^{*}(x;q)=C(x)W_{2}(x;q)$ is the main weight function and $W_{2}(x;q)$ satisfies the equation
\begin{equation}\label{eq:3360}
\frac{W_{2}(qx;q)}{W_{2}(x;q)}=\frac{(q^6-q+1) x^2-q^3+q-1}{q^2 \left(q^2
   x^2-1\right)}.
\end{equation}

Up to a periodic function, a solution of the equation (\ref{eq:3360}) is as
\begin{equation*}
W_{2}(x;q)=\frac{\left(q^2 x^2;q^2\right){}_{\infty }
   \left(q^3-q+1\right)^{\frac{\log
   \left(x^2\right)}{2 \log
   (q)}}}{x^{2} \,\left(\frac{q^6-q+1
  }{q^3-q+1}  x^2;q^2\right){}_{\infty }}=W_{2}(-x;q).
\end{equation*}
Note that
\begin{equation*}
\lim_{q \uparrow 1} W_{2}^{*}(x;q)=\lim_{q \uparrow 1} x^{2} W_{2}(x;q)=x^{2}\sqrt{1-x^{2}},
\end{equation*}
which gives the weight function of the sixth kind Chebyshev polynomials \cite{MR2270049}.

The monic polynomials $\bar{\phi}_{n}(x;q)$ solution of (\ref{eq:368}) satisfy a three term recurrence relation of type (\ref{eq:45}) with
\begin{equation*}
\gamma_{2m}=\frac{q^{2 m+1} \left(q^m-1\right) \left(q^m+1\right)
   \left(\left(q^6-q+1\right) q^{2 m}+\left(-q^3+q-1\right)
   q^2\right)}{\left(q^6-q+1\right)^2 q^{8
   m}-\left(q^2+1\right) \left(q^6-q+1\right) q^{4 m+1}+q^4},
\end{equation*}
and
\begin{equation*}
\gamma_{2m+1}=\frac{q^{2 m} \left(\left(q^3-q+1\right) q^{2 m+1}-1\right)
   \left(\left(q^6-q+1\right) q^{2
   m}-q\right)}{\left(q^6-q+1\right)^2 q^{8
   m+1}-\left(q^2+1\right) \left(q^6-q+1\right) q^{4 m}+q}.
\end{equation*}

In this direction we have
\begin{equation*}
\lim_{q \uparrow 1} \gamma_{n}=\frac{(4 n+6) \sigma_{n}+n (n+1)}{4 (n+1) (n+2)},
\end{equation*}
which coincides with \cite[Eq. (68.1)]{MR2270049}.

Thus, the norm square value of the monic sequence $\bar{\phi}_{n}(x;q)$ is derived as
\begin{equation*}
\int_{-1}^{1} \bar{\phi}_{n}^{2}(x;q) W_{2}^{*}(x;q) d_{q}x=  d_{n}^{2} \int_{-1}^{1} \frac{\left(q^2 x^2;q^2\right){}_{\infty }
   \left(q^3-q+1\right)^{\frac{\log
   \left(x^2\right)}{2 \log
   (q)}}}{\,\left(\frac{q^6-q+1
  }{q^3-q+1}  x^2;q^2\right){}_{\infty }} d_{q}x,
\end{equation*}
where
\begin{multline*}
d_{2m}^{2}=\frac{(1-q) \left(q^5+q^4+q^3-1\right)
   \left(\qnum{6}-2\right) q^{m (2 m-1)-2}
   \left(q^3-q+1\right)^{m+1}}{\left(q
   \left(q^6-q+1\right);q^4\right){}_m} \\
\times \frac{(-1;q)_{m+1} (q;q)_m \left(q
   \left(q^3-q+1\right);q^2\right){}_m
   \left(q^5-1+\frac{1}{q};q^2\right){}_m
   \left(\frac{q^6-q+1}{q^5-q^3+q^2};q^2\right){}_{m+1}}{2
   \left(q^5+q^2-1\right)
   \left(q^5-1+\frac{1}{q};q^4\right){}_m
   \left(q^5-1+\frac{1}{q};q^4\right){}_{m+1}
   \left(\frac{q^6-q+1}{q^3};q^4\right){}_{m+1} },
\end{multline*}
and
\begin{multline*}
d_{2m+1}^{2}=(1-q) \left(q^5+q^4+q^3-1\right)
   \left(\qnum{6}-2\right) q^{2 m^2+m-2}
   \left(q^3-q+1\right)^{m+1} \\
\times \frac{(-1;q)_{m+1} (q;q)_m \left(q
   \left(q^3-q+1\right);q^2\right){}_{m+1}
   \left(q^5-1+\frac{1}{q};q^2\right){}_{m+1}
   \left(\frac{q^6-q+1}{q^5-q^3+q^2};q^2\right){}_{m+1}}{2
   \left(q^5+q^2-1\right)
   \left(\left(q^5-1+\frac{1}{q};q^4\right){}_{m+1}\right){}^2
   \left(\frac{q^6-q+1}{q^3};q^4\right){}_{m+1} \left(q
   \left(q^6-q+1\right);q^4\right){}_{m+1}}.
\end{multline*}

\subsection{Generalized $q$-Hermite polynomials}
Consider the $q$-difference equation
\begin{multline*}
x^{2}\left( \left(1-q^2\right) x^2-1 \right)D_{q} D_{q^{-1}} \phi_{n}(x;q)
+(q+1) x \left(x^2+p\right) D_{q} \phi_{n}(x;q) 
\\+\left(
q \qnum{-n} x^{2} -\sigma_{n} p  
 \right) \phi_{n}(x;q) =0,
\end{multline*}
as a special case of (\ref{eq:17}), with the following polynomial solution
\begin{equation}\label{eq:64n}
\phi_{n}(x;p;q)=x^{\sigma_{n}} \qhyper{2}{1}
{q^{\sigma_{n}-n},0}
{q^{2 \sigma_{n}+1} \left(-p  q^2+p +1\right)}
{q^2}{q^2 \left(1-q^2\right) x^2}.
\end{equation}

The polynomial sequence (\ref{eq:64n}) satisfies an orthogonality relation as
\begin{equation*}
\int_{-\alpha}^{\alpha} W_{3}^{*}(x;p;q) \phi_{n}(x;p;q) \phi_{m}(x;p;q) d_{q}x = 
\int_{-\alpha}^{\alpha} W_{3}^{*}(x;p;q) \phi_{n}^{2}(x;p;q) d_{q}x \delta_{n,m},
\end{equation*}
in which $\alpha=1/\sqrt{1-q^{2}}$, $W_{3}^{*}(x;p;q)=C(x)W_{3}(x;p;q)$ is the main weight function and $W_{3}(x;p;q)$ satisfies the equation
\begin{equation}\label{eq:4460}
\frac{W_{3}(qx;p;q)}{W_{3}(x;p;q)}=\frac{-p  q^2+p +1}{q^2+\left(q^2-1\right) q^4 x^2}.
\end{equation}

Up to a periodic function, a solution of the equation (\ref{eq:4460}) is in the form
\begin{equation*}
W_{3}(x;p;q)=\frac{\left(p  \left(1-q^2\right)+1\right)^{\frac{\log
   \left(x^2\right)}{2 \log (q)}} 
   \qf{q^2 \left(1-q^2\right)
   x^2}{q^{2}}{\infty} }
{x^{2}}=W_{3}(-x;p;q),
\end{equation*}
where
\begin{equation*}
\lim_{q \uparrow 1} W_{3}^{*}(x;p;q)=\lim_{q \uparrow 1} x^{2} W_{3}(x;p;q)=x^{-2p}e^{-x^{2}},
\end{equation*}
appearing the weight function of generalized Hermite polynomials  \cite[Eq. (80)]{MR2421844}.

The monic polynomials type of polynomials (\ref{eq:64n}), $\bar{\phi}_{n}(x;q)$, satisfy a three term recurrence relation of type (\ref{eq:45}) with
\begin{equation*}
\gamma_{2m}=-\frac{\left(p  \left(q^2-1\right)-1\right) q^{2 m-1}
   \left(q^{2 m}-1\right)}{q^2-1},
\end{equation*}
and
\begin{equation*}
\gamma_{2m+1}=\frac{\left(-p  q^2+p +1\right) q^{4 m+1}-q^{2
   m}}{q^2-1},
\end{equation*}
such that
\begin{equation*}
\lim_{q \uparrow 1} \gamma_{n}=\frac{1}{2} \left(p  \left((-1)^n-1\right)+n\right),
\end{equation*}
exactly coincides with \cite[Eq. (79.1)]{MR2270049}.

Consequently, the norm square value corresponding to the the monic type of polynomials (\ref{eq:64n}) takes the form
\begin{multline*}
\int_{-\alpha}^{\alpha} \bar{\phi}_{n}^{2}(x;q)W_{3}^{*}(x;p;q) d_{q}x\\
= d_{n}^{2}  \int_{-\alpha}^{\alpha} \left(p  \left(1-q^2\right)+1\right)^{\frac{\log
   \left(x^2\right)}{2 \log (q)}} 
   \qf{q^2 \left(1-q^2\right)x^2}{q^{2}}{\infty}\, d_{q}x,
\end{multline*}
where
\begin{equation*}
d_{2m}=\frac{1}{2} q^{m (2 m-1)}  
\frac{\left(-p q^2+p +1\right)^m}{\left(q^2-1\right)^{2 m}} (-1;q)_{m+1} (q;q)_m \left(-p  q^3+p
    q+q;q^2\right){}_m,
\end{equation*}
and
\begin{equation*}
d_{2m+1}=\frac{1}{2} q^{m (2 m+1)} 
  \frac{ \left(-p  q^2+p +1\right)^m}{\left(q^2-1\right)^{2 m+1}}  (-1;q)_{m+1} (q;q)_m
   \left(-p  q^3+p  q+q;q^2\right){}_{m+1}.
\end{equation*}

We finally point out that if we choose $p=0$, then the weight function is reduced to the weight function of discrete $q$-Hermite I polynomials,
\begin{equation*}
W(x;0;q)=\frac{1}{\left(\left(1-q^2\right) x^2;q^2\right){}_{\infty}},
\end{equation*}
as well as
\begin{equation*}
\phi_{n}(x;0;q)=c_{n} h_{n}(x\sqrt{1-q^{2}};q),
\end{equation*}
where $h_{n}(x;q)$ denotes the discrete $q$-Hermite I polynomials \cite[Eq. (14.28.1)]{MR2656096}.

\providecommand{\bysame}{\leavevmode\hbox to3em{\hrulefill}\thinspace}
\providecommand{\MR}{\relax\ifhmode\unskip\space\fi MR }
\providecommand{\MRhref}[2]{%
  \href{http://www.ams.org/mathscinet-getitem?mr=#1}{#2}
}
\providecommand{\href}[2]{#2}

\end{document}